\documentclass[12pt]{article}
\usepackage[cp1251]{inputenc}

\usepackage{latexsym}
\usepackage{amsmath}
\usepackage{amsfonts}

\usepackage[russian]{babel}

\usepackage{hyperref}

\newcommand{\termin}[1]{{\it{#1}}}

\def\C{{\mathbb{C}}}
\def\F{{\mathbb{F}}}

\def\R{{\mathcal{R}}}

\def\A{{\mathcal{A}}}
\def\B{{\mathcal{B}}}
\def\Reals{{\mathbb{R}}}
\newtheorem{theorem}{Теорема}
\newtheorem{lemma}{Лемма}

\newenvironment{proof}{{$\rhd$}}{$\Box$\\}

\title{Об алгоритмической неразрешимости проблемы
вложимости алгебраических  многообразий над полем нулевой характеристики}
\author{
{А. Я. Канель-Белов, %\footnote{ BIU,MIPT},
А. А. Чиликов}\footnote{College of Mathematics and Statistics, Shenzhen University, Shenzhen, 518061, China}
\footnote{ BIU}
\footnote{Московский Физико-Технический Институт,
факультет инноваций и высоких технологий,
лаборатория продвинутой комбинаторики и сетевых приложений
}
\footnote{Московский Государственный Технический
Университет им. Н.Э. Баумана,
факультет Информатика и системы управления,
кафедра ИУ-8 Информационная безопасность,
}
\footnote{Passware, Research Department,
chilikov@passware.com
}
}
\begin{document}

\maketitle

%\tableofcontents
%
%\pagebreak

%\makeatletter
%\renewcommand{\@evenhead}{\hfil \copyright \Copyright. Revision \Revision}
%\renewcommand{\@oddhead}{\hfil \copyright \Copyright. Revision \Revision}
%\makeatother

\begin{abstract}

Мы показываем, что для двух аффинных многообразий над произвольным полем
характеристики ноль не существует в общем виде алгоритма проверки наличия
вложения одного алгебраического многообразия в другое. Более того, мы
устанавливаем это для аффинных многообразий, чьи координатные кольца заданы
образующими и определяющими соотношениями.
Более того, одно из этих многообразий можно взять аффинным пространством,
а в случае поля вещественных чисел -- аффинной прямой.

\end{abstract}

\section{Введение}
\label{Intro}

Задача классификации алгебраических многообразий с точностью
до изоморфизма представляется одной из центральных задач
алгебраической геометрии. Чрезвычайно интересной и фундаментальной является
задача об алгоритмической разрешимости проверки наличия изоморфизма.

Близкой задачей является задача о вложимости многообразий.
В общем виде эта задача формулируется так:

\medskip
{\it Пусть $\A$ и $\B$ -- два алгебраических многообразия.
Определить, существует ли вложение $\A$ в $\B$.
}
\medskip

В случае, когда многообразия заданы каким-либо конструктивным способом
(например, системами уравнений и образующими), эта задача естественным
образом приводит к следующей -- придумать алгоритм, позволяющий
по заданным многообразиям установить существование вложения (даже без
его явного построения) или же отсутствие такового. Иными словами,
к вопросу об алгоритмической разрешимости проблемы вложимости.

Вопросы об алгоритмической разрешимости различных классов задач
в неявной форме ставились  еще в XIX веке (например, Десятая проблема
Гильберта о диофантовых уравнениях). Формализация понятия алгоритма,
проведенная в работах Тьюринга, Черча, Геделя и иных
авторов в 30-е годы XX века,
дала толчок к активному исследованию данных вопросов,
но к сожалению это не отразилось в должной мере на алгебраической геометрии.
Одним из наиболее
примечательных результатов, достигнутых в этой области, стало
полученное Ю.~В.~Матиясевичем
отрицательное решение  Десятой проблемы Гильберта
(\cite{Matiyasevich}). Исследования в этой области активно продолжаются
и в наши дни (см. например \cite{Denef}, \cite{ch-nah}, \cite{ch-kur},
\cite{ch-fpm1}, \cite{ch-fpm2}, \cite{ch-yu}, \cite{ch-eng17}).

В данной работе мы рассмотрим вопрос об алгоритмической
разрешимости задачи вложимости многообразий над полями $\Reals$ и $\C$.

Иными словами, о существовании алгоритма, позволяющего определить по
уравнениям, задающим $\A$ и $\B$, существует ли искомое вложение.

Полученные результаты естественным образом обобщаются на многие другие
поля нулевой характеристики.

Мы показываем, что для двух аффинных многообразий над произвольным полем
характеристики ноль не существует в общем виде алгоритма проверки
наличия вложения.

Мы это проверяем для аффинных многообразий, чьи координатные кольца
заданы образующими и определяющими соотношениями.

В частном случае, когда $\A = A$ -- аффинная прямая, соответствующее вложение
задается полиномами от одной переменной. Эти полиномы должны удовлетворять
уравнениям, задающим $\B$. При этом отображения не должны быть константными
(иначе все $\A$ отображается в одну точку $\B$, что не является вложением).
Соответственно, вопрос о существовании вложения эквивалентен вопросу
о существовании у системы уравнений, задающих $\B$, неконстантных решений
в кольце многочленов $\F[t]$ над основным полем $\F$.

В более общем случае, когда $\A = A^m$ -- $m$-мерная аффинная плоскость,
отображение задается полиномами от $m$ переменных. Для того, чтобы оно было
вложением, необходимо чтобы эти полиномы были алгебраически независимыми
(т.е. соответствующее расширение имело степень трансцендентности $m$).
%\todo{тут уточнить, какое ограничение возникает}.

В нашей работе установлено, что проблема вложимости алгоритмически
неразрешима для случая, когда основным полем является поле вещественных
чисел $\Reals$ или поле комплексных чисел $\C$.

Авторы выражают благодарность J. Kollar за предоставленные им новые
результаты об уравнениях Пелля и ценные замечания.

Данная работа была проведена с помощью
Российского Научного Фонда Грант N 17-11-01377.

\section{Предварительные сведения}
\label{Preliminaries}

\subsection{Уравнения Пелля для многочленов}
\label{Pre:Pell}

{\it Уравнением Пелля} над кольцом $\R$ называется уравнение вида
\begin{equation}
X^2 - \lambda Y^2 = 1
\end{equation}
где $\lambda \in \R$ -- параметр, а $X$ и $Y$ -- неизвестные.
Решения также ищутся в кольце $\R$.

Описание множества решений
уравнения Пелля над различными кольцами является интересной и сложной
задачей. В данном разделе мы приведем некоторые сведения об уравнениях
Пелля, поскольку они будут играть центральную роль в наших дальнейших
конструкциях.
Для наших целей будут интересны уравнения Пелля над
кольцом многочленов от одной или нескольких переменных над полем $\F$,
а также один из частных случаев уравнения Пелля -- уравнение
\begin{equation}
\label{eq:pell0}
X^2 - (T^2-1) Y^2 = 1
\end{equation}
где $T \in \R$ -- некоторый заданный параметр.

Решения уравнения $X^2 - \lambda Y^2 = 1$ для
произвольного $\lambda$ образуют абелеву группу относительно операции
$$\circ: (X,Y) \circ (X',Y') = (XX'+ \lambda YY', XY'+X'Y)
$$
В любом решении можно поменять знак как при $X$, так и при $Y$.
Смена знака при $Y$ приводит к получению обратного элемента
(относительно операции $\circ$), а одновременная смена знака при
$X$ и $Y$ -- к умножению
на решение $(-1,0)$, которое само является элементом порядка $2$
в этой группе. Единичным элементом в этой группе является $(1,0)$.

Для случая $\lambda = T^2-1$, где $T$ -- многочлен над $\F$, существует
также очевидное решение $(T,1)$. Применяя к нему операцию $\circ$,
получаем целую серию решений. Если $T$ -- многочлен, отличный от константы,
то все эти решения будут различными.

Значительно менее тривиальным является вопрос о точной структуре группы
решений. Оказывается, что в интересных нам случаях вся она (с точностью
до упомянутой выше замены знаков) порождается
одним единственным элементом (<<примитивным>> решением).

Похожие
утверждения
для различных случаев
доказаны в \cite{Denef}, \cite{Hazama}. Окончательное решение для
важного нам случая $\R = \F[t]$ и $\lambda = T^2-1$, где $T \in \F[t]\backslash \F$
было получено %профессором
J. Kollar и сообщено авторам в частной
переписке (результат готовится к публикации).

%В параграфе 4 работы \cite{Hazama} показано, что указанная группа
%порождается одним элементом (<<примитивным>> решением) с точностью
%до знака.

Имеет место весьма нетривиальная 
\begin{theorem}[J. Kollar]
\label{kollar}
Если $T$ -- многочлен над $F$, не являющийся константой,
то группа (относительно операции $\circ$ решений (\ref{eq:pell0})
(как и в классическом целочисленном случае) порождена элементом $(T,1)$.

Иными словами, множество 
решений уравнения (\ref{eq:pell0}) представимо в виде
$\{\pm 1, 0\} \cup \{(\pm X_N, \pm Y_N)\}$, где
$X_N = \sum\limits_{k=0}^{[N/2]}
\binom{N}{2k} {(T^2-1)}^{k}T^{N-2k}$,
$Y_N = \sum\limits_{k=0}^{[N/2]}
\binom{N}{2k+1} {(T^2-1)}^{k}T^{N-1-2k}$ для некоторого целого
положительного $N$.
\end{theorem}

%\subsection{Многочлены и подстановки}
\subsection{Многочлены и делимость}
\label{Pre:Poly}

%Теперь для завершения конструкции многообразия необходимо явно
%построить полином
%$\hat{P}$, удовлетворяющий условию
%$W_i \mid (\hat{P} - 3i)$.

Нам
потребуется еще несколько вспомогательных утверждений,
связанных с делимостью многочленов от многих переменных.

Основным результатом данного раздела будет построение
полинома $\hat{P}$ и семейства полиномов $W_k$ от нескольких переменных,
удовлетворяющих \termin{условиям делимости} $W_k \mid (\hat{P} - 3k)$.
Эти полиномы будут существенно использованы в построениях
из раздела \ref{Case:Complex} (и не требуются для случая вещественных чисел).

Приводимые здесь рассуждения в целом основаны на идеях,
обсуждавшихся в \cite{lktg-94}, задача 2. Проект на конференции школьников
был основан на задаче, поставленной С.~В.~Конягиным: 
{\it  Дан многочлен $P(x)$ с целочисленными коэффициентами.
верно ли что найдется такое $n$, что все простые делители $P(n)$ меньше $n$.} 
С.~В.~Конягин решил задачу для многочлена вида $ax^n+b$. 
Идея, изложенная ниже позволяет решить ее для квадратного трехчлена. 
Общее решение нам не известно. 

Начнем с наблюдения:

\begin{lemma}
\label{q-mod}
Пусть $P(x_1, \ldots, x_n)$, $Q(x_1, \ldots, x_n)$ -- произвольные
полиномы от $n$ переменных над полем $\F$. Пусть
$$
P^{[1]}(x_1, \ldots, x_n, u) =
P(x_1+uQ(x_1,\ldots,x_n), \ldots, x_n+uQ(x_1,\ldots,x_n))
$$
Тогда существует $R(x_1, \ldots, x_n, u)$, такой, что
$$
P^{[1]}(x_1, \ldots, x_n, u) = P(x_1, \ldots, x_n) +
Q(x_1, \ldots, x_n) R(x_1, \ldots, x_n, u)
$$
\end{lemma}

Из данной леммы вытекает 

%
%\begin{lemma}
\begin{theorem}
\label{p-h}
Пусть $C_1, \ldots, C_m$~-- некоторые константы из произвольного поля $\F$.
Тогда существуют
такие семейства полиномов $H_m(x_1,\ldots,x_m)$,
$P_m(x_1,\ldots,x_m)$
с коэффициентами из $\F$,
для которых одновременно выполняются
условия делимости
$H_k(x_1,\ldots,x_k) \mid ( P_m(x_1, \ldots, x_m) - C_k )$
при всех $k \in \{ 1, \ldots, m \}$, причем $H_m$ и $P_m$
существенно зависят от $x_m$ и при этом $H_s$ при $s<m$ от $x_m$ не зависят.
Оба эти семейства могут быть эффективно построены.
\end{theorem}
%\end{lemma}
%
\begin{proof}
Индукция по $m$~-- числу переменных, $\vec{x}$ обозначает набор $(x_1,\ldots,x_m)$.

\medskip

{\bf База индукции.}
Для $m=1$ утверждение очевидно. Например,
можно положить $H_1(x_1) = x_1$, $P_1(x_1) = x_1+1$.

\medskip

{\bf Индуктивный переход.} Пусть для некоторого $m$ соответствующие
семейства уже построены. Положим $P^{[1]}_{mk} = P_m - C_k$ и
\begin{equation}
%$$
%Q_m(x_1,\ldots,x_m) = \prod\limits_{k=1}^{m} P^{[1]}_{mk}(x_1,\ldots,x_m)
%= \prod\limits_{k=1}^{m} \left(P_m(x_1,\ldots,x_m)-C_k\right)
Q_m(\vec{x}) = \prod\limits_{k=1}^{m} P^{[1]}_{mk}(\vec{x})
= \prod\limits_{k=1}^{m} \left(P_m(\vec{x})-C_k\right)
%$$
\end{equation}
Рассмотрим полином
\begin{equation}
%$$
\tilde{P}_m(x_1,\ldots,x_m,u) =
%P(x_1+uQ_m(x_1,\ldots,x_m), \ldots, x_m+uQ_m(x_1,\ldots,x_m))
P(x_1+uQ_m(\vec{x}), \ldots, x_m+uQ_m(\vec{x}))
%$$
\end{equation}
В силу леммы \ref{q-mod}
$\tilde{P}_m(\vec{x},u)$
%$\tilde{P}_m(x_1,\ldots,x_m,u)$
представляется в виде
\begin{equation}
%$$
P_m(\vec{x})
+ Q_m(\vec{x}) R_m(\vec{x},u)
%P_m(x_1,\ldots,x_m)
%+ Q_m(x_1,\ldots,x_m) R_m(x_1,\ldots,x_m,u)
%$$
\end{equation}
Обозначая $\tilde{P}^{[1]}_{mk} = \tilde{P}_m - C_k$, получаем для него аналогичное
представление
\begin{equation}
%$$
\tilde{P}^{[1]}_{mk}(\vec{x},u) = P^{[1]}_{mk}(\vec{x})
+ Q_m(\vec{x}) R_m(\vec{x},u)
%\tilde{P}^{[1]}_{mk}(x_1,\ldots,x_m,u) = P^{[1]}_{mk}(x_1,\ldots,x_m)
%+ Q_m(x_1,\ldots,x_m) R_m(x_1,\ldots,x_m,u)
%$$
\end{equation}
По предположению индукции для $P_m$ выполнены условия делимости:
$H_k \mid P^{[1]}_{mk}$.
Поскольку $Q_m$ также делится на $P^{[1]}_{mk}$, имеем
\begin{equation}
\label{p'''-short}
\tilde{P}^{[1]}_{mk}(\vec{x},u) =
P^{[1]}_{mk}(\vec{x}) R'_{mk}(x_1,\ldots,x_m,u)
%\tilde{P}^{[1]}_{mk}(x_1,\ldots,x_m,u) =
%P^{[1]}_{mk}(x_1,\ldots,x_m) R'_{mk}(x_1,\ldots,x_m,u)
\end{equation}
где
\begin{equation}
\label{r'-short}
R'_{mk}(\vec{x},u) =
1 + R_m(\vec{x},u) \cdot
\frac{Q_m(\vec{x})}{P^{[1]}_{mk} (\vec{x})}
%R'_{mk}(x_1,\ldots,x_m,u) =
%1 + R_m(x_1,\ldots,x_m,u) \cdot
%\frac{Q_m(x_1,\ldots,x_m)}{P^{[1]}_{mk} (x_1,\ldots,x_m)}
%= P^{[1]}_m(x_1,\ldots,x_m) R'_m(x_1,\ldots,x_m,u)
\end{equation}
Это означает, что полином
$\tilde{P}_m(\vec{x},u)$
%$\tilde{P}_m(x_1,\ldots,x_m,u)$
делится на
$P^{[1]}_{mk}$, а следовательно, и на $H_k$ (при всех $k \in \{1, \ldots, m\}$).

Таким образом, сконструированный нами полином $\tilde{P}_m$ (от $m+1$ переменной)
выглядит подходящим кандидатом в качестве нового $P_{m+1}$.
Осталось лишь выбрать подходящий $H_{m+1}$, и добиться для выполнения
условия делимости для $k=m+1$.

Запишем искомое условие делимости: $H_{m+1} \mid \tilde{P}^{[1]}_{m,m+1}$. В силу
формул (\ref{p'''-short}) и (\ref{r'-short}) получаем
\begin{equation}
%$$
\tilde{P}^{[1]}_{m,m+1}(\vec{x},u) =
P^{[1]}_{m,m+1}(\vec{x}) R'_{m,m+1}(\vec{x},u)
%\tilde{P}^{[1]}_{m,m+1}(x_1,\ldots,x_m,u) =
%P^{[1]}_{m,m+1}(x_1,\ldots,x_m) R'_{m,m+1}(x_1,\ldots,x_m,u)
%$$
\end{equation}
где
\begin{equation}
%$$
R'_{m,m+1}(\vec{x},u) =
1 + R_m(\vec{x},u) \cdot
\frac{Q_m(\vec{x})}{P^{[1]}_{m,m+1} (\vec{x})}
%R'_{m,m+1}(x_1,\ldots,x_m,u) =
%1 + R_m(x_1,\ldots,x_m,u) \cdot
%\frac{Q_m(x_1,\ldots,x_m)}{P^{[1]}_{m,m+1} (x_1,\ldots,x_m)}
%$$
\end{equation}
Для выполнения условий делимости достаточно положить
\begin{equation}
\left\{
\begin{array}{c}
%$$
H_{m+1}(x_1,\ldots,x_m,x_{m+1}) = R'_{m,m+1}(x_1,\ldots,x_m,x_{m+1})
%$$
\\
%и
%$$
P_{m+1}(x_1,\ldots,x_m,x_{m+1}) = \tilde{P}_{m}(x_1,\ldots,x_m,x_{m+1})
%$$
\end{array}
\right.
\end{equation}
Из построения ясно, что $H_{m+1}$ и $P_{m+1}$
(в отличии от $H_s$ при $s<m+1$) существенно зависят
от переменной $x_{m+1}$.

Дополненные семейства полиномов будут удовлетворять условиям делимости
при всех $k \in \{ 1, \ldots, m+1\}$, что и доказывает утверждение леммы.
\end{proof}
Теперь мы можем легко построить искомые полиномы $\hat{P}$ и $W_k$,
удовлетворяющие условию $W_k \mid (\hat{P} - 3k)$ при всех
$k \in \{1, \ldots, n\}$. Действительно,
для этого достаточно положить $C_k = 3k$, построить по теореме
\ref{p-h} семейства
$P_k$, $H_k$ длины $n$, и положить $W_k( x_1, \ldots, x_k ) =
H_k( x_1, \ldots, x_k )$ и $\hat{P}( x_1, \ldots, x_n ) =
P_n( x_1, \ldots, x_n )$.

%Таким образом, если $T_{\hat{j}} = \pm 1$ при некотором значении
%$\hat{j}$, то соответствующая компонента имеет размерность $d$. Однако
%при этом все остальные компоненты нульмерны. Иными словами, в этом
%случае общая размерность многообразия не превосходит $d$.
%
%Рассмотрим теперь второй случай. Пусть при некотором $\hat{j}$
%имеет место $T_{\hat{j}} = C_{\hat{j}} \neq \pm 1$. Соответствующая
%компонента многообразия имеет размерность $0$. Более того,
%в силу
%леммы \ref{reverse} получаем $T_{\J} = C_j$ при  $j <\hat{j}$.
%Соответствующие $\hat{j}-1$ компонент многообразия также нульмерны.
%
%В случае $\deg T_j > 0$ рассмотрение аналогично предыдущему параграфу.
%Компонента многообразия будет параметризоваться целочисленным
%набором $N_{1j}, \ldots, N_{dj}$ для которых соответствующие решения
%для $X_{ij}$, $Y_{ij}$, $Z_{ij}$, $U_{ij}$, $V_{ij}$ строятся явно.
%Соответствующая компонента имеет размерность $1$.

\section{Случай вещественных чисел}
\label{Case:Real}

Нашей ближайшей целью будет установление алгоритмической неразрешимости
задачи о вложимости многообразий над полем $\Reals$. Для этого
мы рассмотрим более частный вопрос о вложимости аффинной прямой
в заданное многообразие, и сконструируем специальный класс многобразий,
для представителей которого невозможно установить существование
вложения по определяющим соотношениям.

Результаты данного параграфа сформулированы для случая основного
поля $\Reals$, но естественным образом обобщаются на случай
произвольного
упорядоченного
поля нулевой характеристики.

Здесь и далее мы будем существенно опираться на классические результаты
о неразрешимости диофантовых уравнений \cite{Matiyasevich}.
Назовем
\termin{семейством полиномов Матиясевича} семейтво полиномов
$Q(\sigma_1, \ldots, \sigma_{\tau}, x_1,\ldots, x_s)$ для которого
задача о существовании решения при данном наборе параметров полинома
алгоритмически неразрешима. Как установлено
в \cite{Matiyasevich}, такое семейство полиномов существует.

Теперь перейдем непосредственно к построению искомого класса многообразий.
Для этого нам потребуется несколько вспомогательных конструкций.

Рассмотрим аффинное пространство размерности $5d+1$. Координаты в этом
пространстве назовем $X_i,Y_i,Z_i,U_i,W_i$ ($1 \le i \le d$) и $T$.
Зададим многообразие $B_{(d)}$ при помощи системы образующих и соотношений.
\begin{equation}
\label{system:basic}
\left\{
\begin{array}{c}
X_i^2 - (T^2-1) Y_i^2 = 1
\\
Y_i - (T-1) Z_i = V_i
\\
V_i U_i = 1
\end{array}
\right.
\end{equation}
где $1 \le i \le d$.

Исследуем решения этой системы.
%Исследуем решения этой системы в $\Reals[t]$.

Очевидно, что при фиксированном значении $i$ допустимые значения
координат $X_i$, $Y_i$, $Z_i$, $U_i$, $W_i$ определяются
одним и тем же значением $T$.
Таким образом, многообразие представляет собой объединение <<слоев>>,
соответствующих различным значениям $T$. При этом каждый из этих
слоев является прямой суммой $d$ одинаковых многообразий, а именно,
допустимых решений <<короткой>> системы:
\begin{equation}
\label{system:short}
\left\{
\begin{array}{c}
X^2 - (T^2-1) Y^2 = 1
\\
Y - (T-1) Z = V
\\
V U = 1
\end{array}
\right.
\end{equation}
Решения же системы (\ref{system:short}) полностью описываются
следующим утверждением:
\begin{lemma}
\label{l:system:short}
Для любого решения системы (\ref{system:short}) выполняется:
\begin{enumerate}
\item
$U$ и $V$ -- ненулевые константы в $\F[t]$ ($\deg U = \deg V = 0$);
\item
Для $T$, $X$ и $Y$ возможны три случая:
\begin{enumerate}
\item
$T = \pm 1$, при этом $X = \pm 1$ и $Y$ -- любое;
\item
$T$ -- константа, отличная от $\pm 1$, при этом $X$ и $Y$ -- также
некоторые подходящие константы из $\F$ (при этом значение $X$ определяет
значение $Y$ с точностью до знака, и наоборот);
\item
\label{system:short:poly}
$T$ -- полином, отличный от константы, и при этом $X=\pm 1$, $Y=0$
или $X = \pm \hat{X}_N$, $Y = \pm \hat{Y}_N$, где
$$
\left\{
\begin{array}{c}
\hat{X}_N( T ) = \sum\limits_{k=0}^{[N/2]}
\binom{N}{2k} {(T^2-1)}^{k}T^{N-2k}
\\
\hat{Y}_N( T ) = \sum\limits_{k=0}^{[N/2]}
\binom{N}{2k+1} {(T^2-1)}^{k}T^{N-1-2k}
\end{array}
\right.
$$
для некоторого целого
положительного $N$.
\end{enumerate}
%либо $T= \pm 1$ и $X = \pm 1$, либо $Y = \sum\limits_{k=0}^{[N/2]}
%\binom{N}{2k+1} {(T^2-1)}^{k}T^{N-1-2k}$ для некоторого целого $N$.
\end{enumerate}
\end{lemma}
\begin{proof}
Первая часть очевидна.
%Во второй же части первый вариант проверяется
%прямым подсчетом. Поэтому достаточно разобрать лишь случай
%$T \neq \pm 1$.

Вторая часть утверждения в случае \ref{system:short:poly} непосредственно
следует из теоремы \ref{kollar}.

В случае $T = \pm 1$ первое уравнение системы не накладывает
вообще никаких ограничений на $Y$. Второе уравнение, в свою очередь,
влечет $Y = ( T-1 )Z + V$. При любом выборе $V \in \F \setminus \{ 0 \}$
и $Z \in \F[t]$ соответствующее решение существует и единственно.

Если же $T$ -- константа, отличная от $\pm 1$
% $\deg T = 0$ (т.е. $T \neq \pm 1$ -- константа из $\F$),
то прямое вычисление %(??)
показывает, что $Y$ и $Z = ( Y-V ) / (T-1)$
также являются константами.

Более того, допустимое множество значений $X$ и $Y$ в этом случае
описывается ранее найденной последовательностью. Однако теперь значения
будут уже не полиномами, а константами.

\end{proof}

%Несложно заметить, что
Таким образом,
структура полученного множества существенно зависит
от $T$.

Для случая \ref{system:short:poly}
%$T \neq \pm 1$
заметим еще один важный факт.
\begin{lemma}
Если $\deg T > 0$, то $V = Y \bmod (T-1) = N$ для некоторого целого $N$
и $Z = ( Y-N ) / (T-1)$.
%В противном случае $Y$ и $Z$ -- константы в $\F[t]$.
\end{lemma}
\begin{proof}
При $\deg T > 0$ уравнение $Y - (T-1) Z = V$ в совокупности с
ранее показанным условием $\deg V = 0$ означает, что
остаток от деления многочлена $Y$ на многочлен $T$ есть в точности
константа $V \in \F$.  Легко видеть,
что в выражении для $Y$ все слагаемые делятся на $T-1$ без остатка,
за исключением слагаемого, отвечающего $k=0$. Оно равно
$\binom{N}{1} T^{N-1} = N\bmod (T-1)$.
Таким образом, значение $V$
в любом решении исходной системы должно быть целым числом $N$.
Прямой подсчет показывает, что $Z = ( Y-N ) / (T-1)$.
\end{proof}
Таким образом, слои построенного множества $\B_{(d)}$ обладают
понятной структурой. Возможны три случая.

\begin{enumerate}
\item
При $\deg T > 0$ каждому набору
целых чисел $N_i$ соответствуют решения $Y_i$ и $X_i$,
являющиеся многочленами, и определеные с точностью до знака,
а также константы $V_i = N_i$ и $U_i = 1/V_i$ и
$Z_i = ( Y_i-V_i ) / (T-1)$. %Размерность равна $1$.
\item
При $\deg T = 0$ и $T \neq \pm 1$
имеем константные решения для $Y_i$, выбираемые из заданной
последовательности. Значения $X_i$, $Z_i$, $V_i$ и $U_i$
также являются константами и определяются выбранными значениями
$Y_i$. %Размерность равна $0$.
\item
При $T = \pm 1$ получаем  $X_i = \pm 1$, для произвольно выбранных
констант $V_i$ и полиномов $Z_i$ определяем $U_i = 1/V_i$ и
$Y_i = (T-1) Z_i + V_i$. %Размерность равна $d$.
\end{enumerate}

Все вышеприведенные рассуждения справедливы для произвольного
основного поля $\F$. Теперь используем специфику случая $\F = \Reals$.

Перейдем к пространству размерности $5d+2$, обозначим новую координату
через $S$, и дополним основную систему (\ref{system:basic}) уравнением
\begin{equation}
\label{eq:square}
T = S^2+2
\end{equation}
Далее заметим, что из $T = S^2 +2$ следует, что случай $T = \pm 1$
невозможен. Таким образом, все общие решения систем
(\ref{system:basic}) и (\ref{eq:square}) либо являются константами
(случай $\deg T = 0$, $T \neq \pm 1$), либо отвечают некоторому
целочисленному набору параметров $(N_1,\ldots,N_d)$.
Назовем первые решения <<плохими>>, а вторые -- <<хорошими>>.

Рассмотрим семейство полиномов Матиясевича
$Q(\sigma_1, \ldots, \sigma_{\tau}, x_1, \ldots, x_s)$.
Пусть $d \le s$.
Тогда, добавив новое уравнение
\begin{equation}
\label{eq:matiy}
Q_{\sigma}(V_1, \ldots, V_s) = 0
\end{equation}
к системам (\ref{system:basic}) и (\ref{eq:square}),
мы получаем систему, задающую новое многообразие. Обозначим его
$\B'_{(d), \sigma}$.

Если $Q_{\sigma} = 0$ не имеет целочисленных решений,
то у исходной системы нет <<хорших>> решений. Тогда многообразие
$\B'_{(d), \sigma}$ нульмерно, и вложений из $A$ в $\B'_{(d)}$ не существует.

В противном случае для каждого решения $N_1, \ldots N_s$ строятся
явно
%функции
неконстанные полиномы $S(t)$,
$T(t)$, $Y_i(t)$, $X_i(t)$ и $Z_i(t)$, являющиеся решениями.
Например, можно положить $S(t) = t$, $T(t) = t^2+2$, и определить $X_i$,
$Y_i$ и $Z_i$ по формулам для случая \ref{system:short:poly} из
леммы \ref{l:system:short}. Также однозначно определяются константы
$V_i = N_i$ и $U_i = 1/N_i$.
В совокупности эти функции
задают вложение прямой в многообразие $\B'_{(d), \sigma}$.

Поскольку задача о существовании целочисленных решений для $Q_{\sigma}$
алгоритмически неразрешима, то алгоритмически неразрешима и задача
о вложимости $A$ в $\B'_{(d), \sigma}$ (в частности, в $\B'_{(s), \sigma}$).
Входом при этом являются уравнения, задающие $\B'_{(d), \sigma}$.

Таким образом, нами доказана
\begin{theorem}
Задача о вложимости аффинной прямой над $\Reals$ в произвольное
алгебраические многобразие $\B$ (заданное образующими и соотношениями)
алгоритмически неразрешима.
\end{theorem}
Отсюда сразу следует:
%
%\begin{corollary}
\begin{theorem}
Задача о вложимости произвольного алгебраического многобразия $\A$ над
$\Reals$ в произвольное
алгебраические многобразие $\B$ %(оба заданы образующими и соотношениями)
алгоритмически неразрешима.
\end{theorem}
%\end{corollary}
%

\section{Случай комплексных чисел}
%(т.е. произвольного поля нулевой характеристики)}
\label{Case:Complex}

В комплексном случае ситуация сложнее, чем в вещественном.
В самом деле, рассуждение для $\Reals$ существенно опирается на тот
факт, что уравнение $S^2+a$ не имеет решений при $a > 0$
(отюда следует, что $T = S^2+2 \neq \pm 1$). Однако в поле комплексных
чисел это неверно. Таким образом, мы не можем исключить случай
$T = \pm 1$ и $X_i = \pm 1$. В этом случае $Y_i$
может быть задано, вообще говоря, произвольным образом.

Поэтому в случае основного поля $\C$ мы немного усложним рассуждение.
А именно, рассмотрим вопрос о вложимости аффинного пространства
$A^m$ в заданное многообразие $\B$. Нашей целью будет, как и раньше,
сконструировать
такой класс многообразий, для представителей которого
(при некотором подходящем $m$)
невозможно установить существование искомого
вложения по определяющим соотношениям.

Результаты данного параграфа сформулированы для случая основного
поля $\C$, но
%естественным
проходят и для 
произвольного
%алгебраически замкнутого
поля нулевой характеристики, поскольку все коэффициенты у наших полиномов целые.

Зададим алгебраическое многообразие $\B_{(d,e)}$ при помощи системы образующих и соотношений.
\begin{equation}   \label{system:huge}
\left\{
\begin{array}{c}
X_{ij}^2 - (T_j^2-1) Y_{ij}^2 = 1
\\
Y_{ij} - (T_j-1) Z_{ij} = V_{ij}
\\
V_{ij} U_{ij} = 1
\\
%T_{j+1} = \prod\limits_{k=1}^{j} \left( (T_k^2-1) W_k \right) W_{j+1}^{m_{j+1}}
T_{j+1} = \prod\limits_{k=1}^{j} \left( (T_k^2-1) W_k \right) W_{j+1}
\\
T_1 = \hat{P}( W_1, \ldots, W_n )
\end{array}
\right.
\end{equation}
где $1 \le i \le d$, $1 \le j \le e$.%, $\{m_j\}_{j=1}^e$~достаточно
%быстро растущая последовательность целых чисел.
При этом полином $\hat{P}$
выберем таким образом, чтобы $W_j$ был делителем
$\hat{P}(W_1,\ldots,W_n)-3j$.
Полиномы с нужными свойствами были построены в разделе \ref{Pre:Poly}
(см. лемму \ref{p-h}).
%В дальнейшем этот полином будет
%построен явно (см. теорему \ref{p-h}).

% AC: вроде это не является необходимым!!!
%Последовательность $\{m_j\}_{j=1}^e$ выбирается так, чтобы в случае,
%когда все переменные входящие в $T_1$ были бы различными, многочлены $T_j$
%были бы алгебраически независимыми.
%Для этого надо убедится в том, что якобиан
%$\left(\partial T_j/\partial x_k\right)$
%имел ранг $n$, что легко обеспечить неравенствами $m_1\ll m_2\ll\dots\ll m_k$
%(ибо произведение всех $m_i$ возникает только в одном члене детерминанта.

%
%
Иными словами, мы <<клонируем>> основную систему (\ref{system:basic})
из предыдущего
параграфа в большом числе экземпляров, и дополняем ее <<связующими>>
соотношениями между параметрами $T_j$. Исследуем решения этой системы
в $\C[t_1,\ldots,t_m]$.

Совершенно очевидно, что соотношения для $X_{ij}$, $Y_{ij}$, $Z_{ij}$,
$U_{ij}$, $V_{ij}$ при фиксированном $T_j$ аналогичны ранее рассмотренным.
Иными словами, при фиксированном наборе $T_j$ множество решений
есть прямая сумма слоев $\B_{(d)}$, ранее уже изученных. Однако при различных
$T_j$ соответствующие слои ведут себя по-разному.
% (в частности, иогут быть
%как одномерными, так и нульмерными многообразиями).
При этом наличие связей
между $T_j$ говорит о том, что возможно, не все варианты допустимы.

Как показывают результаты предыдущего раздела, для каждого $j$ имеют место
следующие важные случаи:
\begin{enumerate}
\item
$T_j = \pm 1$;
\item
$\deg T_j = 0, T_j \neq \pm 1$;
\item
$\deg T_j > 0$.
\end{enumerate}

Далее в тексте мы будем обозначать через $C_i$ константы из основного поля
$\C$ (в частности, полиномы степени $0$).

%Наиболее важным случаем с точки зрения <<исключений>> является случай
%$T_{\hat{j}} = \pm 1$ при некотором значении $\hat{j}$. В этом случае
%$T_{\hat{j}}^2-1 = 0$ и
%при всех $j > \hat{j}$ получаем
%$T_{j} = \prod\limits_{k=1}^{j-1} \left( (T_k^2-1) W_k \right) W_j$.

Для дальнейших рассуждений удобно сформулировать и доказать еще одно
вспомогательное утверждение:

\begin{lemma}      \label{reverse}
Пусть для некоторого $N$ выполнено
$T_N = C_N \neq 0$.
Тогда все $W_k$
при $k \le N$
и все $T_k$ при при $k \le N-1$ являются константами.
\end{lemma}

\begin{proof}
Заметим, что произведение нескольких многочленов из $\F[t]$
является ненулевой константой в том и только в том случае, когда
все сомножители являются также ненулевыми константами.
Теперь докажем утверждение по индукции.

\medskip
{\bf База индукции.} При $N = 0$ имеем
$T_1 = W_1 = C_1 \neq 0$ и утверждение очевидно.
\medskip

{\bf Индуктивный переход.} Пусть для $N-1$ утверждение верно, и пусть
$T_N = C_N \neq 0$. Поскольку
$T_N = \prod\limits_{k=1}^{N-1} \left( (T_k^2-1) W_k \right) W_N = C_N \neq 0$,
сразу получаем что $W_N$ -- ненулевая константа. Также все
$(T_k^2-1)$ являются ненулевыми константами. Следовательно,
и все $T_k$ -- тоже константы. В частности, $T_{N-1}$ -- константа.
Если при этом $T_{N-1} \neq 0$, то утверждение леммы для $N$ следует
из утверждения леммы для $N-1$. Таким образом, осталось показать,
что $T_{N-1} \neq 0$.

Но если $T_{N-1} = 0$, то поскольку $T_{N} = T_{N-1} (T_{N-1}^2-1) W_N$,
то и $T_N = 0$, что невозможно. Таким образом, лемма \ref{reverse} доказана.
\end{proof}

%$\prod\limits_{k=1}^{\hat{j}-1} (T_k^2-1) W_j = C$,
%С другой стороны, поскольку $$ разрешая в обратном порядке уравнения
%$T_{\hat{j}} = \prod\limits_{k=1}^{\hat{j}-1} \left( (T_k^2-1) W_k \right) W_j = C$,
%где $C $

Теперь необходимо рассмотреть несколько случаев.

\medskip
{\bf{Первый случай:}} среди $T_j$ нет констант. В этом случае, все компоненты
слоя соответствуют ранее изученным множествам $\B_{(d)}$ (для первого случая
из предыдущего параграфа). Они параметризуются
целочисленными наборами $N_{ij}$ независимо друг от друга. Слой является
прямой суммой соответствующих компонент.

\medskip
{\bf{Второй случай:}} среди $T_j$ есть ненулевые константы, ни одна из которых
не равна $\pm 1$. Пусть $\hat{j}$ -- максимальный из индексов этих констант.
Тогда в силу леммы \ref{reverse} при $j <\hat{j}$ мы имеем $T_{j} = C_j$.
Иными словами, все предыдущие $T_j$ -- также константы. Одновременно
с этим и $W_j$ при $j \le \hat{j}$ также будут константами.

В частности,
$T_1 = \hat{P}(W_1,\ldots,W_n)$ также будет ненулевой константой $C_1$.
Поскольку при этом все $W_k$ являются делителями
$\hat{P}(W_1,\ldots,W_n)-3k = C_1-3k$, то все они, за исключением,
возможно, одного значения, также являются ненулевыми константами.

Этим единственным значением, очевидно, должно быть $\hat{j}+1$
(разумеется, при $\hat{j} < e$). Действительно, если $\hat{j} < e$
и $W_{\hat{j}+1}$ -- ненулевая константа, то и
$T_{\hat{j}+1} = T_{\hat{j}} (T_{\hat{j}}^2-1) W_{\hat{j}+1}$
также будет ненулевой константой ($T_{\hat{j}} \neq \pm 1$),
что противоречит выбору $\hat{j}$. Далее возможны еще два случая:
$W_{\hat{j}+1} = 0$ или же $W_{\hat{j}+1}$ -- неконстантный полином.
В первом случае все последующие $T_j$ будут нулями, а во втором
они будут однозначно определяться значениями констант $W_j$
и полиномом $W_{\hat{j}+1}$.
%Таким образом, размерность слоя
%будет равна $1$ при $\hat{j} < e$ или же $0$ при $\hat{j} = e$.

\medskip
{\bf{Третий случай:}} среди $T_j$ есть константы, хотя бы одна из которых
равна $\pm 1$. Обозначим соответствующий индекс через $\hat{j}$.
%Наиболее важным случаем с точки зрения <<исключений>> является случай
%$T_{\hat{j}} = \pm 1$ при некотором значении $\hat{j}$.
В этом случае
$T_{\hat{j}}^2-1 = 0$ и
при всех $j > \hat{j}$ получаем
$T_{j} = \prod\limits_{k=1}^{j-1} \left( (T_k^2-1) W_k \right) W_j = 0$.

В силу леммы \ref{reverse} при $j <\hat{j}$ мы имеем $T_{j} = C_j$.
При этом $C_j \neq \pm 1$ (иначе $C_{j+1} = 0$). Аналогично и все
$W_j$, за исключением, возможно, одного значения, также являются
ненулевыми константами.

%Таким образом, как и во втором случае, почти все компоненты нульмерны.
%Исключением будет компонента $\hat{j}+1$ (при $\hat{j} < e$), имеющая
%размерность $1$, а также компонента $\hat{j}$, которой соответствует
%множество $\B_{(d)}$ (для третьего случая из предыдущего параграфа).
%Это множество имеет размерность $d$. Таким образом, общая размерность
%слоя не превосходит $d+1$.

Рассмотрим  семейство {\it полиномов Матиясевича}
%$Q(\sigma_1, \ldots, \sigma_{\tau}, x_1, \ldots, x_s)$.
$Q(\sigma, x_1, \ldots, x_s)$.
Проблема существования решения диофантова
%уравнение $Q(\sigma_1, \ldots, \sigma_{\tau}, V_{1j}, \ldots, V_{sj}) = 0$
уравнения $Q(\sigma, V_{1j}, \ldots, V_{sj}) = 0$
алгоритмически неразрешима.
Пусть $d \le s$.
Тогда, добавив новые уравнения $Q(\sigma, V_{i1}, \ldots, V_{is}) = 0$
к системе (\ref{system:huge}),
мы получаем систему, задающую новое многообразие. Обозначим его
$\B'_{(d,e),\sigma}$.

\begin{lemma}
\label{l:no-embed}
Если $Q_{\sigma} = 0$ не имеет целочисленных решений,
то при $m \ge d+2$
%не существует вложения $\A = A^m$ в $\B'_{(d,e)} = \B'_{(s-2,e)}$.
не существует вложения $\A = A^m$ в $\B'_{(d,e),\sigma}$.
\end{lemma}
\begin{proof}
Пусть искомое вложение существует. Тогда оно задается некоторой
системой полиномов от $m$ переменных. Соответствующие полиномы
будут решениями системы (\ref{system:huge}). Значения координат
$U_{ij}$, $V_{ij}$ будут константами.

Поскольку $Q_{\sigma} = 0$ не имеет целочисленных решений,
то у исходной системы нет решений для которых $\deg T_1 > 0$.
Тогда возможные решения соответствуют второму либо третьему из
расмотренных выше случаев. Далее рассмотрим их отдельно.

Если имеет место
второй случай, то все $W_i$, за исключением, быть может, одного,
являются константами. Значения $W_i$ однозначно парамеризуют
все значения $T_j$. Значения $T_j$ в этом случае однозначно
определяют все остальные кординаты. Таким образом, образ
соответствующего отображения имеет размерность $1$. Следовательно,
при $m \ge 2$ оно не может быть вложением.

В третьем случае также все $W_i$, за исключением, быть может, одного,
являются константами. Значения $W_i$ снова однозначно парамеризуют
все значения $T_j$. Однако при $T_{\hat{j}} = \pm 1$ для компоненты
с номером $\hat{j}$ мы находимся в условиях третьего случая из
раздела \ref{Case:Real}. В этом случае $Z_{i\hat{j}}$ могут быть
выбраны любыми, остальные же переменные однозначно ими определяются.
Еще одна свободная переменная $W_{\hat{j}+1}$ может возникнуть
по тем же соображениям, что и в предыдущем случае. Таким образом, образ
соответствующего отображения имеет размерность $\le d+1$.
Следовательно, при $m \ge d+2$ оно не может быть вложением.
\end{proof}

\begin{lemma}
\label{l:embed}
Если $Q_{\sigma}$ имеет целочисленные решения, то существует
вложение аффинного пространства  $\A=A^s$ в многообразие
$\B'_{(d,s),\sigma}$.
\end{lemma}
\begin{proof}
%
%Если же 
Поскольку $Q_{\sigma}$ имеет целочисленные решения, то для каждого такого
решения $N_1, \ldots, N_s$ по любому заданному набору $T_j$, $1 \le j \le s$,
могут быть
построены
%строятся .....
явно функции $U_{ij}(T_j)$, $V_{ij}(T_j)$,
$Y_{ij}(T_j)$, $X_{ij}(T_j)$ и $Z_{ij}(T_j)$, являющиеся решениями.
Для окончания построения нужно осталось показать, что можно подобрать
подходящие $T_j$ (и $W_j$), чтобы выполнялись связывающие их уравнения.

Заметим, что $T_j$ однозначно определяются по $W_j$. Сами же $W_j$
нами уже построены при помощи полиномов $H_j$ из теоремы \ref{p-h}.
При этом каждый из $H_j$ зависит от переменных $x_1, \ldots, x_j$.
Таким образом, задавая произвольную параметризацию для $x_j$ мы получим
подходящий полином $H_j$ (а значит и $W_j$). Таким образом,
$W_j$ алгебраически независимы. Следовательно, в совокупности
с остальными координатными функциями, они задают некоторое вложение
аффинного пространства  $\A=A^s$ в многообразие $\B'_{(d,s),\sigma}$.
\end{proof}

Заметим, что при доказательстве отсутствия вложения в лемме \ref{l:no-embed}
мы наложили некоторое дополнительное условие на $d$, но не на $e$.
Напротив, для существования вложения (при наличии решений у $Q_{\sigma}$)
в лемме \ref{l:embed}
ограничения накладываются на $e$, но не на $d$. Это дает нам возможность
согласовать параметры. Для этого достаточно положить $m=e=s$ и $d=s-2$.
В этом случае $A^s$ вложимо в $\B'_{(s-2,s),\sigma}$, если $Q_{\sigma}$ имеет
целочисленные решения, и не вложимо в противном случае.

%Следовательно,
%вложение в $\B'_{(d,e)} = \B'_{(s-2,e)}$ существует, если
%$e \ge m$. В частности, $A^s$ можно вложить в $\B'_{(s-2,s)}$,
%а $A^m$ (при $m \ge s$) -- в $\B'_{(s-2,m)}$.
%
%...

Поскольку задача о существовании целочисленных решений для $Q_{\sigma}$
алгоритмически неразрешима, то алгоритмически неразрешима и задача
о вложимости $A^s$ в $\B'_{(s-2,s),\sigma}$ (входом при этом являются
уравнения,
задающие $\B'_{(s-2,s),\sigma}$).

Таким образом, нами доказана
\begin{theorem}
Существует целое положительное $s$, для которого задача о вложимости
аффинного пространства $A^s$ над $\C$ в
алгебраические многобразие $\B'_{(s-2,s),\sigma}$
% (заданное образующими и соотношениями)
алгоритмически неразрешима.
\end{theorem}
Отсюда сразу следуют:
\begin{theorem}
Существует целое положительное $s$, для которого задача о вложимости
аффинного пространства $A^s$ над $\C$ в произвольное
алгебраические многобразие $\B$ (заданное образующими и соотношениями)
алгоритмически неразрешима.
\end{theorem}
и
%
%\begin{corollary}
\begin{theorem}
Задача о вложимости произвольного алгебраического многобразия $\A$ над
$\C$ в произвольное
алгебраические многобразие $\B$ %(оба заданы образующими и соотношениями)
алгоритмически неразрешима.
\end{theorem}
%\end{corollary}
%

\section{Заключение}

В работе показана алгоритмическая неразрешимость проблемы
вложимости двух алгебраических многообразий над полями вещественных чисел
$\Reals$ и комплексных чисел $\C$. Построен класс многообразий, для
представителей которого алгоритмически неразрешима проблема существования
вложения аффинной прямой над $\Reals$. Построен класс многообразий, для
представителей которого алгоритмически неразрешима проблема существования
вложения аффинного пространства размерности $m$ над $\C$ для достаточно
большого $m$.

Результаты для $\C$ обобщены на случай произвольного поля
нулевой характеристики,
а для $\Reals$ -- на случай произвольного упорядоченного поля
нулевой характеристики.

% Created at Dec'2018
% by Alexey Chilikov
% chilikov@passware.com

% This file contains a bibliography
% of ...

\newcommand{\by}[1]{{\it{#1}~}}
\newcommand{\paper}[1]{{\rm{#1}. }}

\def \journ{}
\def \jour{}
\def \book{}
\def \yr{}
\def \vol{Vol}
\def \no{№}
\def \pages{p}

\end{document}